\documentclass[reqno, 12pt]{amsart}

\usepackage{amsmath,amssymb,amsthm}
\usepackage[T1]{fontenc}
\usepackage{lmodern}
\usepackage[linktocpage,colorlinks,linkcolor=magenta,citecolor=magenta,urlcolor=black,pagebackref=false]{hyperref}
\usepackage[noEnd=false]{algpseudocodex}
\usepackage{enumitem}
\usepackage[all]{xy}
\usepackage{microtype}
\usepackage{xcolor}
\usepackage{cleveref}  

\let\oldphi\phi
\let\phi\varphi
\let\varphi\oldphi
\let\oldepsilon\epsilon
\let\epsilon\varepsilon
\let\varepsilon\oldepsilon

\newcommand{\mbb}[1]{\ensuremath{\mathbb{#1}}}
\newcommand{\Z}{\mbb{Z}}
\newcommand{\N}{\mbb{N}}

\newcommand{\mbf}[1]{\ensuremath{\mathbf{#1}}}
\newcommand{\Xb}{\mbf{X}}
\newcommand{\Pb}{\mbf{P}}
\newcommand{\Fb}{\mbf{F}}
\newcommand{\Gb}{\mbf{G}}
\newcommand{\Ib}{\mbf{I}}
\newcommand{\Jb}{\mbf{J}}
\newcommand{\Mb}{\mbf{M}}
\newcommand{\Nb}{\mbf{N}}
\newcommand{\Ab}{\mbf{A}}
\newcommand{\Lb}{\mbf{L}}

\newcommand{\mcal}[1]{\ensuremath{\mathcal{#1}}}
\newcommand{\Cc}{\mcal{C}}

\newcommand{\tn}[1]{\textnormal{#1}}
\let\hom\relax\newcommand{\hom}{\tn{Hom}}
\let\ker\relax\newcommand{\ker}{\tn{ker}}
\newcommand{\im}{\tn{im}}
\newcommand{\orb}{\tn{Orb}}
\newcommand{\id}{\tn{id}}
\newcommand{\rank}{\tn{rank}}
\newcommand{\lcm}{\tn{lcm}}
\newcommand{\lt}{\tn{lt}}
\newcommand{\lm}{\tn{lm}}
\newcommand{\lc}{\tn{lc}}

\newcommand{\syz}{\tn{Syz}}
\newcommand{\mon}{\tn{Mon}}

\newtheorem{thm}{\bf Theorem}[section]
\newtheorem{lem}[thm]{\bf Lemma}
\newtheorem{cor}[thm]{\bf Corollary}
\newtheorem{prop}[thm]{\bf Proposition}

\theoremstyle{definition}
\newtheorem{defn}[thm]{\bf Definition}

\newtheorem{rem}[thm]{\bf Remark}
\newtheorem{ex}[thm]{\bf Example}
\newtheorem{algo}[thm]{\bf Algorithm}
\newtheorem{proc}[thm]{\bf Procedure}

\newcommand{\la}{\langle}
\newcommand{\ra}{\rangle}
\newcommand{\FI}{\tn{FI}}
\newcommand{\OI}{\tn{OI}}

\renewcommand{\mod}{\tn{-Mod}}

\renewcommand{\l}{\ell}
\renewcommand{\th}{^{th}}

\numberwithin{equation}{section}
\newenvironment{algorithm}[1][] 
{\begin{algo}[#1]\mbox{}
\par\vspace{-0.5em}
\noindent\rule{\textwidth}{0.5pt}}
{\stepcounter{thm}
\par\vspace{-0.5em}\noindent\rule{\textwidth}{0.5pt}
\end{algo}}
\algrenewcommand\algorithmicrequire{\textbf{Input:}}
\algrenewcommand\algorithmicensure{\textbf{Output:}}

\title{Computing Gr\"obner bases and free resolutions of OI-modules}

\author{Michael Morrow}
\address{Department of Mathematics, University of Kentucky, 715 Patterson Office Tower, Lexington,
KY 40506 USA}
\email{michael.morrow@uky.edu}

\author{Uwe Nagel}
\address{Department of Mathematics, University of Kentucky, 715 Patterson Office Tower, Lexington,
KY 40506 USA}
\email{uwe.nagel@uky.edu}

\begin{document}
\begin{abstract}
Given a sequence of related modules $M_n$ defined over a sequence of related polynomial rings, one may ask how to simultaneously compute a finite Gr\"obner basis for each $M_n$. Furthermore, one may ask how to simultaneously compute the module of syzygies of each $M_n$. In this paper we address both questions. Working in the setting of OI-modules over a Noetherian polynomial OI-algebra, we provide OI-analogues of Buchberger's Criterion, Buchberger's Algorithm for computing Gr\"obner bases, and Schreyer's Theorem for computing syzygies. We also establish a stabilization result for Gr\"obner bases. 
\end{abstract}
\maketitle

\tableofcontents


\section{Introduction}
\label{section:intro}
Computing Gr\"obner bases and syzygies is a classical problem in commutative and computational  algebra. Recall that if $R$ is a Noetherian polynomial ring over a field and $F$ is a free module over $R$ of finite rank, then a finite Gr\"obner basis of any finitely generated submodule $M$ of $F$ can be computed using Buchberger's Algorithm. Furthermore, a theorem of Schreyer shows how to determine a finite Gr\"obner basis for the module of syzygies of $M$. These results can be used to give a constructive proof of Hilbert's Syzygy Theorem which asserts that $M$ has a finite free resolution. We refer the reader to \cite[Chapter 15]{Eis95} for more on this topic. In this paper we extend these results to a setting where we consider a sequence  $(M_n)_{n \in \Z_{\ge 0}}$ of related modules $M_n$ over related polynomial rings whose number of variables increases with $n$. We show that there is a finite algorithm that simultaneously computes Gr\"obner bases of all modules in the sequence. 

Gr\"obner bases in contexts beyond Noethertian polynomial rings have been considered previously, mainly for ideals (see, e.g., \cite{BD, C,  HKL, HS}). Typically they focus on ideals in polynomial rings with infinitely many variables. However, in order to treat also modules a more abstract approach is beneficial. 
In \cite{NR19}, Nagel and R\"omer introduced OI-modules over OI-algebras, which may be viewed as a generalization of OI-modules over a fixed ring studied in \cite{SS-14}. In particular, Noetherian polynomial OI-algebras are classified in \cite{NR19}. Furthermore, it is shown that any submodule $\Mb$ of a finitely generated free OI-module $\Fb$ over a Noetherian polynomial OI-algebra $\Pb$ has a finite Gr\"obner basis. It follows that $\Fb$ is Noetherian.  The Gr\"obner basis result has also been used, for example in \cite{Nag21},  where it is shown that the equivariant Hilbert series of any finitely generated graded OI-module over $\Pb$ is rational. 

The argument for the existence of a finite Gr\"obner basis in  \cite{NR19} is not constructive. We address the question of computing such a Gr\"obner basis in \Cref{section:buchberger}. 
In particular, we give a criterion for determining when a subset of $\Mb$ forms a Gr\"obner basis of 
$\Mb$ (see Theorem \ref{theorem:buchbergercriterion}), and use it to describe a finite algorithm for computing a finite Gr\"obner basis (see Algorithm \ref{algorithm:buchberger}). A key ingredient for our finiteness results is a combinatorial statement about simultaneously factoring OI-morphisms (Lemma \ref{lemma:oifactorization}). Any OI-module $\Mb$ may be thought of as a sequence $(\Mb_n)_{n \in \Z_{\ge 0}}$ of $\Pb_n$-modules $\Mb_n$. It follows that our algorithm for $\Mb$ produces  a Gr\"obner basis for every $\Mb_n$. This is possible because stabilization occurs in the sense that  there is an integer $n_0$ such that the   Gr\"obner basis of $\Mb_n$ can be determined from the  Gr\"obner basis of $\Mb_{n_0}$ if $n \ge n_0$. Deciding a priori when stabilization occurs, i.e.\ estimating $n_0$, is a very interesting and open problem. We establish a condition that guarantees stabilization has occurred (see \Cref{prop:stabilization}). 

We also consider the problem of computing syzygies. Given a finite Gr\"obner basis $G$ of a submodule $\Mb$ of a finitely generated free $\Pb$-module $\Fb$, it is shown in Section \ref{section:schreyer} how to determine a finite Gr\"obner basis for the module of syzygies of $\Mb$ in terms of $G$ (see \Cref{theorem:oi-schreyers-theorem}). This may be seen as an extension of Schreyer's result that solves this problem in the classical setting. 

The finiteness results in \cite{NR19} were used to establish that a finitely generated $\Pb$-module $\Mb$ has a free resolution with finitely generated free $\Pb$-modules. However, only a few explicit resolutions are known, see \cite{FN21,FN22} for families of examples. Furthermore, if $\Mb$ is a graded OI-module it is shown in \cite{FN21}  that $\Mb$ has a graded minimal free resolution $\Fb^\bullet$.  Here minimality means that in every homological degree $d$ the rank of $\Fb^d$ is at most the rank of $\Gb^d$ in any other graded free resolution $\Gb^\bullet$ of $\Mb$. By \cite{FN21}, such a minimal free resolution is unique up to isomorphism of complexes. 
In Section \ref{section:resolution} we apply our results about Gr\"obner bases and syzygies to give a procedure for computing free resolutions. In particular, if $\Mb$ is graded, we show how to compute a graded minimal free resolution of $\Mb$ up to any chosen homological degree.


\section{Preliminaries}
\label{section:prelim}
In this section we recall the necessary background on OI-modules and fix notation. 
Let $K$ be a Noetherian commutative ring with identity. 

\begin{defn}
Let OI be the category of totally ordered finite sets and order-preserving injective maps. An OI-\emph{algebra} over $K$ is a (covariant) functor from OI to the category of commutative, associative, unital $K$-algebras.
\end{defn}

To define a functor out of OI, it is enough to define it on the corresponding skeleton. Therefore, we regard OI as the category whose objects are intervals of the form $[n]=\{1,\ldots,n\}$ for $n\in\Z_{\geq0}$ (we put $[0]=\emptyset$) and whose morphisms are order-preserving injective maps 
$[m] \to [n]$. 

If $\Ab$ is any functor out of OI, we typically write $\Ab_n$ in place of $\Ab([n])$. We refer to $\Ab_n$ as the \emph{width $n$ component} of $\Ab$. If $\epsilon$ is an OI-morphism, we often write $\epsilon_*$ instead of $\Ab(\epsilon)$. 
We abuse notation and write $\hom(m,n)$ for the set of OI-morphisms $\hom_{\OI}([m],[n])$ from $[m]$ to $[n]$. 

Polynomial OI-algebras were introduced in \cite[Definition 2.17]{NR19}. This paper also identifies the Noetherian algebras among these. Following the notation of \cite{NR19}, their building block is the following algebra. 

\begin{defn}
Define an OI-algebra $\Xb^{\OI,1}$ by letting
\[
\Xb^{\OI,1}_n = \begin{cases}
K&\text{ if }n=0\\
K[x_\pi\;:\;\pi\in\hom(1,n)]&\text{ if }n>0
\end{cases}
\]
and, for $\epsilon\in\hom(m,n)$, defining $\epsilon_*:\Xb^{\OI,1}_m\to\Xb^{\OI,1}_n$ via $x_\pi\mapsto x_{\epsilon\circ\pi}$. 
\end{defn}

By \cite[Proposition 4.8 and Theorem 6.15]{NR19}), apart from trivial direct summands,  any Noetherian polynomial OI-algebra is isomorphic to $(\Xb^{\OI,1})^{\otimes c}$ for some integer $c > 0$. Identifying each $\pi$ with its image in $[n]$, one may intuitively  think of $\Pb=(\Xb^{\OI,1})^{\otimes c}$ as a sequence $(\Pb_n)_{n \in \Z_{\ge 0}}$ of  polynomial rings, where  
\[
\Pb_n = K\begin{bmatrix}
x_{1,1}  & \cdots & x_{1,n}\\
\vdots & \ddots & \vdots \\
x_{c,1} & \cdots & x_{c,n}
\end{bmatrix}
\]
has $cn$ variables, together with $K$-algebra homomorphisms $\epsilon_* \colon \Pb_m \to \Pb_n$ induced by any 
$\epsilon\in\hom(m,n)$ and defined by $x_{i,j}\mapsto x_{i,\epsilon(j)}$.

\begin{rem}
Replacing the category OI by the category FI that has the same skeleton as OI, but whose morphisms are injective maps $[m] \to [n]$, one defines FI-algebras and polynomial FI-algebras 
$(\Xb^{\FI,1})^{\otimes c}$ analogously (see \cite[Definition 2.4]{NR19} for details). Every FI algebra is also an OI-algebra. In \cite{DEF}, FI-algebras that are finitely generated in width one are studied. These are precisely the quotients of polynomial FI algebras 
$(\Xb^{\FI,1})^{\otimes c}$. 
\end{rem}

We say an OI-algebra $\Ab$ is \emph{graded} if each $\Ab_n$ is a graded $K$-algebra and each $\epsilon_* \colon \Ab_m\to\Ab_n$ is a graded homomorphism of degree 0. Any algebra 
$\Pb=(\Xb^{\OI,1})^{\otimes c}$ is an example of  a graded OI-algebra by assigning each variable 
degree 1. Throughout this note, we always use this grading of $\Pb$.

We now turn to OI-modules over OI-algebras.

\begin{defn}[\cite{NR19}]
An OI-\emph{module} $\Mb$ over an OI-algebra $\Ab$ is a (covariant) functor from OI to the category of $K$-modules such that
\begin{enumerate}
\item each $\Mb_n$ is an $\Ab_n$-module, and
\item for each $a\in\Ab_m$ and $\epsilon\in\hom(m,n)$ we have a commuting diagram
\begin{equation}
\label{diagram:oimodule}
\xy\xymatrixrowsep{10mm}\xymatrixcolsep{10mm}
\xymatrix {
\Mb_m\ar[d]_{a\cdot}\ar[r]^{\Mb(\epsilon)} & \Mb_n\ar[d]^{\Ab(\epsilon)(a)\cdot}\\
\Mb_m\ar[r]^{\Mb(\epsilon)} & \Mb_n
}
\endxy
\end{equation}
where the vertical maps are multiplication by the indicated elements.
\end{enumerate}
We sometimes refer to $\Mb$ as an $\Ab$-\emph{module}.
\end{defn} 

Intuitively, one may think of an $\Ab$-module $\Mb$ as a sequence $(\Mb_n)_{n \in \Z_{\ge 0}}$ of $\Ab_n$-modules $\Mb_n$ together with maps $\epsilon_* \colon \Mb_m \to \Mb_n$ induced by any $\epsilon \in \hom(m, n)$. 

A \emph{homomorphism} of $\Ab$-modules 
is a natural transformation $\phi:\Mb\to\Nb$ such that each $\phi_n \colon \Mb_n\to\Nb_n$ is an $\Ab_n$-module homomorphism.  We sometimes call $\phi$ an $\Ab$-\emph{linear map}. $\Ab$-modules together with $\Ab$-linear maps form an abelian category denoted $\OI\mod(\Ab)$. Thus, all concepts such as subobject, quotient object, kernel, cokernel, injection, and surjection are defined ``width-wise'' from the corresponding concepts in the category of $K$-modules (see \cite[A.3.3]{W}). For example, the \emph{direct sum} of OI-modules $\Mb,\Nb$   is the OI-module $\Mb\oplus\Nb$ defined width-wise by
\[
(\Mb\oplus\Nb)_n=\Mb_n\oplus\Nb_n
\]
for every $n\in\Z_{\geq0}$.  
%
Similarly, if $\phi \colon \Mb\to\Nb$ is an $\Ab$-linear map then the kernel of $\phi$ is a submodule of $\Mb$ defined by $(\ker(\phi))_n=\ker(\phi_n)$ for $n\in\Z_{\geq0}$. The image of $\phi$ is a submodule of $\Nb$ defined analogously.

Note that any OI-algebra $\Ab$ is trivially an OI-module over itself.  An \emph{ideal} of $\Ab$ is a submodule of $\Ab$. 

If $f\in\Mb_n$ for some $n\in\Z_{\geq0}$ then we call $f$ an \emph{element} of $\Mb$ and write $f\in\Mb$. In this case we also say that $f$ \emph{has width $n$}, denoted $w(f)=n$. 
A \emph{subset} $S$ of $\Mb$, denoted $S\subseteq\Mb$, is a subset of the disjoint union $\coprod_{n\geq0}\Mb_n$. The submodule of $\Mb$ \emph{generated by} a subset $S\subseteq\Mb$ is defined to be the smallest submodule of $\Mb$ containing $S$ and is denoted $\la S\ra_{\Mb}$. 
For any integer $m\geq0$ we define the $m$-\emph{orbit} of $S$ as the set
\[
\orb(S,m)=\{\Mb(\epsilon)(s)\;:\;s\in S\cap\Mb_\l,\;\epsilon\in\hom(\l,m)\}\subseteq\Mb_m.
\]
Thus, for every $m\geq0$ we have
\[
(\la S\ra_{\Mb})_m=\la\orb(S,m)\ra, 
\]
where the right-hand side is generated as an $\Ab_m$-submodule of $\Mb_m$. 

\begin{ex}
   \label{exa:ideal}
Consider $\Pb = (\Xb^{\OI,1})^{\otimes 2}$ and $f = x_{1,1} x_{2,2} - x_{1,2} x_{2,1}  \in \Pb_2$. Let $\Ib$ be the ideal of $\Pb$ generated by $f$. Then $\Ib_n$ is the ideal of $\Pb_n$ that is generated by the 2-minors of a generic $2 \times n$ matrix. 

Let now $\Jb$ be the ideal of $\Pb$ that is generated by $g = x_{1,1} x_{2,2} - x_{1,2} x_{2,1}  \in \Pb_3$. Then $\Jb_n$ is equal to $\Ib_{n-1} \Pb_n$, the extension ideal of $\Ib_{n-1}$ in $\Pb_n$. In particular, we have $\Jb_2 = 0$ whereas $\Ib_2 = \langle f \rangle \neq 0$. 

Finally, let $\Lb$ be the ideal of $\Pb$ that is generated by $x_{1,1} x_{2,3} - x_{1,3} x_{2,1}  \in \Pb_3$. Then $\Lb_n$ is the ideal of $\Pb_n$ that is generated by the 2-minors of a generic $2 \times n$ matrix using columns $i, j$ with $j \ge i+2$. 
\end{ex}

We now define the most important class of OI-modules for our purposes.
\begin{defn}[\cite{NR19}]
For any   integer $d\geq0$, define an OI-module $\Fb^{\OI,d}$ over an OI-algebra $\Ab$ as follows:  For $n\in\Z_{\geq0}$ let
\[
\Fb^{\OI,d}_n=\bigoplus_{\pi\in\hom(d,n)}\Ab_n e_\pi\cong (\Ab_n)^{\binom{n}{d}}.  
\]
For $\epsilon\in\hom(m,n)$,  define $\Fb^{\OI,d}(\epsilon) \colon \Fb^{\OI,d}_m\to\Fb^{\OI,d}_n$ via $e_\pi\mapsto e_{\epsilon\circ\pi}$. 

A \emph{free} OI-module over $\Ab$ is an OI-module $\Fb$ isomorphic to a direct sum 
$\bigoplus_{\lambda\in\Lambda}\Fb^{\OI,d_\lambda}$ for integers $d_\lambda\geq0$. If $|\Lambda|=n<\infty$, then $\Fb$ is said to have \emph{rank $n$}. 

To stress the dependence on $\Ab$ we also write $\Fb^{\OI,d}_{\Ab}$  for $\Fb^{\OI,d}$. 
\end{defn}

The fact that rank is well-defined is a special case of \cite[Lemma 3.6]{FN21}. Note that $\Fb^{\OI,0}_{\Ab}$ is simply the OI-algebra $\Ab$ considered as an OI-module over itself.

\begin{rem}Let $\Fb=\bigoplus_{i=1}^s\Fb^{\OI,d_i}$ be a finitely generated,  free OI-module.
\begin{enumerate}
\item For all $n\geq0$ we have
\[
\Fb_n=\bigoplus_{\substack{\pi\in\hom(d_i,n)\\1\leq i\leq s}}\Ab_ne_{\pi,i}
\]
where the second index on $e_{\pi,i}$ is used to keep track of which summand it lives in.
\item $\Fb$ is generated as an OI-module by all $e_{\id_{[d_i]},i}$. We call these the \emph{basis elements} of $\Fb$, and any $e_{\pi,i}$ is said to be a \emph{local basis element} of $\Fb$.
\item Let $f=\sum a_ie_{\pi_i,k_i}\in\Fb_m$ and let $\epsilon\in\hom(m,n)$. By Condition \eqref{diagram:oimodule}, we have
\[
\Fb(\epsilon)(f)=\sum\epsilon_*(a_i)e_{\epsilon\circ\pi_i,k_i}.
\]
\item Any $\Ab$-linear map out of $\Fb$ is determined by the images of the elements $e_{\id_{[d_i]},i}$.
\end{enumerate}
\end{rem}

\begin{ex}
Let $\Pb=\Xb^{\OI,1}$, so that $\Pb_n=K[x_1,\ldots,x_n]$ for $n\in\Z_{\geq0}$. Consider $\Fb=\Fb^{\OI,1}_{\Pb}\oplus\Fb^{\OI,2}_{\Pb}$. Since $\hom(1,n)$ and $\hom(2,n)$ consist of $\binom{n}{1}$ and $\binom{n}{2}$ maps, respectively, we have
\begin{align*}
\Fb_0&=0,\\
\Fb_1 &= K[x_1],\\
\Fb_2 &= K[x_1,x_2]^2\oplus K[x_1,x_2]\cong K[x_1,x_2]^3,\\
\Fb_3 &= K[x_1,x_2,x_3]^3\oplus K[x_1,x_2,x_3]^3\cong K[x_1,x_2,x_3]^6, 
\end{align*}
and in general, $\Fb_n \cong K[x_1,\ldots,x_n]^{\binom{n}{1}+\binom{n}{2}}$.
\end{ex}

An OI-module $\Mb$ over a graded OI-algebra $\Ab$ is  said to be \emph{graded} if each $\Mb_m$ is a graded $\Ab_m$-module and each $\Mb(\epsilon) \colon \Mb_m\to\Mb_n$ is a graded homomorphism of degree 0. If $\Ab$ is graded, then $\Fb^{\OI,d}_{\Ab}$ is a graded module by assigning the element $e_{\id_{[d]}}$ degree $0$. 

We can alter the grading of an OI-module as follows. Given a graded OI-module $\Mb$, we define the $d\th$ \emph{twist (or shift) of} $\Mb$ to be the OI-module $\Mb(d)$ that is isomorphic to $\Mb$ as an OI-module, and whose grading is determined by
\[
[\Mb(d)_n]_j=[\Mb_n]_{d+j}.
\]
If $\Mb$ and $\Nb$ are graded $\Ab$-modules, a morphism $\phi \colon \Mb \to \Nb$ is said to be \emph{graded} if each $\phi_n$ is graded of degree 0.


\section{Gr\"obner Bases and the OI-Buchberger's Algorithm}
\label{section:buchberger}

For the rest of the paper, we assume that $K$ is a field and consider OI-modules over a polynomial OI-algebra  $\Pb=(\Xb^{\OI,1})^{\otimes c}$ with some $c \ge 1$. It has been shown in \cite[Theorem 6.14]{NR19} (see \cite[Theorem 5.6]{Nag21}) that every submodule of a finitely generated free $\Pb$-module has a finite Gr\"obner basis. 
The goal of this section is to produce an algorithm that computes such a Gr\"obner basis in finitely many steps. 

We begin by recalling needed concepts. 
Let $\Fb=\bigoplus_{i=1}^s\Fb^{\OI,d_i}_{\Pb}$ be a free $\Pb$-module with basis $\{e_{\id_{[d_i]},i} \; : \; i \in [s]\}$.
A \emph{monomial} in $\Fb$ is an element of the form $ae_{\pi,i}$ where $a$ is a monomial in $\Pb$. We say the monomial $ae_{\pi,i}$ \emph{involves the basis element }$e_{\id_{[d_i]},i}$. The set of all monomials in $\Fb$ is denoted $\mon(\Fb)$. A \emph{monomial submodule of $\Fb$} is a submodule of $\Fb$ generated by monomials. 

We need to order the monomials in $\Fb$. 

\begin{defn}[cf. \cite{NR19,Nag21}]
\label{definition:monomialorder}
A \emph{monomial order on $\Fb$} is a total order $<$ on $\mon(\Fb)$ such that if $\mu<\nu$ for any monomials $\mu,\nu\in\Fb_m$,  one has
\begin{enumerate}
\item $\mu<a\mu<a\nu$ for any monomial $a \neq 1$ in $\Pb_m$, and
\item $\mu<\Fb(\epsilon)(\mu)<\Fb(\epsilon)(\nu)$ for every $\epsilon\in\hom(m,n)$ with $m<n$.
\end{enumerate}
\end{defn}

For example, if $c = 1$, i.e.\ $\Pb = \Xb^{\OI,1}$ and $\Fb = \Fb^{\OI, 0} \cong \Pb$ consider $x_1 e_{\pi} \in \Fb_5$ and $x_1 e_{\rho} \in \Fb_3$. Note that $\pi$ maps the empty set to $[5]$, whereas  
$\rho$ maps the empty set to $[3]$. Choose any map $\epsilon \in \hom(3,5)$  with $\epsilon (1) = 1$. Then Condition (ii) in \Cref{definition:monomialorder} gives $x_1 e_{\rho} < \Fb(\epsilon)(x_1 e_{\rho}) = x_{\epsilon (1)} e_{\epsilon \circ \rho} = x_1 e_{\pi}$. 

Monomial orders exist.

\begin{ex}[Lexicographic Order]
\label{example:lex}
Order the monomials in each $\Pb_m$ lexicographically with $x_{i',j'}<x_{i,j}$ if either $i'<i$ or $i'=i$ and $j'<j$. Define $e_{\rho,j}<e_{\pi,i}$ if $i<j$. For a fixed $i$, identify a monomial $e_{\pi,i}\in\Fb_m$ with a vector $(m,\pi(1),\ldots,\pi(d_i))\in\N^{d_i+1}$ and order such monomials using the lexicographic order on $\N^{d_i+1}$. Finally, for monomials $ae_{\pi,i}$ and $be_{\rho,j}$ in $\Fb$, define $be_{\rho,j}<ae_{\pi,i}$ if $e_{\rho,j}<e_{\pi,i}$ or $e_{\pi,i}=e_{\rho,j}$ and $b<a$ in $\Pb$. One checks that this gives indeed a monomial order on $\Fb$.
\end{ex}

There is a suitable generalization of Dickson's lemma in our context. Following \cite[Definition 6.1]{NR19}, a monomial $be_{\sigma,j}$ is said to be OI-\emph{divisible}  by a monomial $ae_{\pi,i}$ if $i=j$ and there is an OI-morphism $\epsilon$ such that $\Fb(\epsilon)(ae_{\pi,i})=\epsilon_*(a)e_{\epsilon\circ\pi,i}$ divides $be_{\sigma,j}$. It was shown in \cite{NR19} (see also \cite[Proposition 5.3]{Nag21}) that OI-divisibility gives a well-partial-order on $\mon(\Fb)$. Since any monomial order on $\Fb$ refines the OI-divisibility partial order on $\mon(\Fb)$ it follows that any monomial order $<$ on $\Fb$ is a  \emph{well-order} on $\mon(\Fb)$, i.e. any nonempty subset of $\mon(\Fb)$ has a unique  minimal element with respect to $<$. 

%
The following concepts are defined as in the classical situation.   

\begin{defn}
Fix a monomial order $<$ on $\Fb$. Let $0\not=f\in\Fb$ and write $f=\sum c_i\mu_i$ for nonzero $c_i\in K$ and distinct $\mu_i\in\mon(\Fb)$. The \emph{leading monomial} of $f$ is $\lm(f)=\max_{<}\{\mu_i\}$. If $\lm(f)=\mu_i$ then the \emph{leading coefficient} of $f$ is $\lc(f)=c_i$. The \emph{leading term} of $f$ is $\lt(f)=\lc(f) \cdot \lm(f)$. Given any subset $E\subseteq\Fb$ we set $\lm(E)=\{\lm(f)\;:\;f\in E\text{ is  nonzero}\}$.
\end{defn}

\begin{ex} 
   \label{rem:lead monomial}
Consider elements $a = a_1 + \cdots + a_t \in \Pb$ and $f = f_1 + \cdots + f_s \in \Fb$ with monomials $a_i \in \Pb$ and $f_j \in \Fb$.  Possibly after reindexing, assume $f_1 > f_2 > \cdots >  f_s$ and  $a_1 f_1 > a_2 f_1 > \cdots > a_t  f_1$. It follows that  $a_1 f_1> a_i f_j$ if $(i, j) \neq (1,1)$, and so $\lm (a f) = a_1 f_1 = a_1 \lm (f)$. 
\end{ex}

The properties of a monomial order on $\Fb$ imply the following useful facts. 

\begin{rem}
\label{rem:lm}
For any nonzero  $f,g\in\Fb_m$ and $\epsilon\in\hom(m,n)$, one has:
\begin{enumerate}
\item $\lm(f+g)\leq\max(\lm(f),\lm(g))$,
\item $\lm(af)=a\lm(f)$ for any monomial $a\in\Pb_m$, and
\item $\Fb(\epsilon)(\lm(f))=\lm(\Fb(\epsilon)(f))$.
\end{enumerate}
\end{rem}

We now define our primary object of study.

\begin{defn}[\cite{NR19, Nag21}]
Fix a monomial order $<$ on $\Fb$ and let $\Mb$ be a submodule of $\Fb$. A subset $G\subseteq \Mb$ is called a \emph{Gr\"obner basis} (with respect to $<$) of $\Mb$ if 
\[
\la\lm(\Mb)\ra_{\Fb}=\la\lm(G)\ra_{\Fb}.
\]
\end{defn}
Every submodule of $\Mb$ of $\Fb$ has a Gr\"obner basis by taking $G=\Mb$. We are interested in determining a finite Gr\"obner basis, which is guaranteed to exist. 
 
For the remainder of this section,  we fix a monomial order $<$ on $\Fb$. By restricting $<$ to $\Fb_n$ we obtain a monomial order on the free $\Pb_n$-module $\Fb_n$ in the classical sense (see, e.g.,\cite{Eis95}), which we denote by $<_n$. The following result relates Gr\"obner bases of OI-modules with Gr\"obner bases of their width $n$ components.

\begin{lem}
\label{lemma:widthwisegroebner}
Let $\Mb$ be a submodule of $\Fb$. For any subset $G\subseteq\Mb$, set $G_n=\orb(G,n)$ for $n\in\Z_{\geq0}$. Then $G$ is a Gr\"obner basis of $\Mb$ if and only if each $G_n$ is a Gr\"obner basis of $\Mb_n$ with respect to $<_n$.
\end{lem}

\begin{proof}
Suppose that $G$ is a Gr\"obner basis of $\Mb$ and let $h\in\Mb_n$. Then $\lm_{<_n}(h)=\lm(h)\in\la\lm(G)\ra_{\Fb}$ so that $\lm_{<_n}(h)$ is divisible by some $\Fb(\epsilon)(\lm(g))=\lm_{<_n}(\Fb(\epsilon)(g))$ with $g\in G$ and $\epsilon\in\hom(w(g),n)$. 
As $\Fb(\epsilon)(g)$ is in $G_n$, this 
shows that $\lm_{<_n}(\Mb_n)\subseteq\la\lm_{<_n}(G_n)\ra$. Thus, $G_n$ is a Gr\"obner basis for $\Mb_n$.
\par
The argument for the reverse implication is similar and  left to the reader.
\end{proof}

Combined with \cite[Lemma 2.3]{NR19}, the above observation says that, if we regard $\Mb$ as a sequence of modules $(\Mb_n)_n$, then computing a finite Gr\"obner basis $G$ of $\Mb$ is the same as ``simultaneously'' computing a Gr\"obner basis for each module $\Mb_n$ with the additional property that the Gr\"obner basis of $\Mb_n$ is obtained ``combinatorially'' by the OI-action from the Gr\"obner basis of $\Mb_w$  for each $n \ge w$, where $w$ is the maximum width of an element in $G$. 

The basis for our computational results is a division algorithm. 

\begin{defn}
   \label{def:remainder} 
Let $f\in\Fb_n$ and $B\subseteq\Fb$. An element $r\in\Fb_n$ is called a \emph{remainder of $f$ modulo $B$} if $f=\sum a_iq_i + r$ for some $q_i\in\orb(B,n)$ and $a_i\in\Pb_n$ such that
\begin{enumerate}
\item[(R1)] either $r=0$ or $\lm(r)$ is not OI-divisible by any element of $\lm(B)$,
\item[(R2)] $\lm(r)<\lm(f)$ provided $r,f\not=0$ and $r\not=f$, and
\item[(R3)] $\lm(a_iq_i)\leq\lm(f)$ whenever $a_iq_i,f\not=0$.
\end{enumerate}
\end{defn}

Every element $f \in \Fb$ has  a remainder modulo $B$. If $f$ has width $n$ such a remainder 
 can be computed by applying the classical division algorithm in $\Fb_n$  to $f$ and the set $\orb(B,n)$ with respect to the  monomial order $<_n$ on $\Fb_n$ (see, e.g, \cite[Ch. 15]{Eis95}).  

\begin{rem}
\label{rem:r3}
(i)
Note that Condition (R3) implies  $\lm(uq_i)\leq\lm(f)$ for any monomial $u$ in $a_i$, see \Cref{rem:lead monomial}.  

(ii)
An important property of remainders is that if $f\in\Fb_m$ has a remainder of zero modulo $B$ and $\epsilon\in\hom(m,n)$, then $\Fb(\epsilon)(f)$ also has a remainder of zero modulo $B$. This follows from Remark \ref{rem:lm}(iii) and Definition \ref{definition:monomialorder}(ii).
\end{rem}

We now introduce an important combinatorial result about OI-morphisms that will serve as a key ingredient for our finiteness arguments.

\begin{lem}[OI-Factorization Lemma]
\label{lemma:oifactorization}
Consider any maps $\sigma\in\hom(k_1,m)$ and $\tau\in\hom(k_2,m)$ with  $k_1,k_2,m\in\Z_{\geq0}$ and $k_1,k_2\leq m$. Set $\l=|\im(\sigma)\cup\im(\tau)|$. Then there are maps $\bar{\sigma}\in\hom(k_1,\l)$, $\bar{\tau}\in\hom(k_2,\l)$ and $\rho\in\hom(\l,m)$ such that
\[
\sigma=\rho\circ\bar{\sigma}\quad\text{and}\quad\tau=\rho\circ\bar{\tau} 
\]
as well as $\l=|\im(\bar{\sigma})\cup\im(\bar{\tau})|$. 
\end{lem}

\begin{proof}
Let $L=\im(\sigma)\cup\im(\tau)$ so that $\l=|L|$. Define a map $\gamma \colon L\to[\l]$ via $i\mapsto|L\cap[i]|$. Clearly $\gamma$ is strictly increasing. Since $L$ and $[\l]$ are finite sets of the same cardinality, it follows that $\gamma$ is a bijection. Now define maps
\begin{align*}
&\bar{\sigma} \colon [k_1]\to[\l]\quad\text{by}\quad i\mapsto\gamma(\sigma(i)),\\
&\bar{\tau} \colon [k_2]\to[\l]\quad\text{by}\quad i\mapsto\gamma(\tau(i)),\\
\text{and }&\rho \colon [\l]\to[m]\quad\text{by}\quad i\mapsto\gamma^{-1}(i).
\end{align*}
Since $\gamma$, $\sigma$ and $\tau$ are strictly increasing, so are $\bar{\sigma}$, $\bar{\tau}$ and $\rho$. We have $\sigma=\rho\circ\bar{\sigma}$ and $\tau=\rho\circ\bar{\tau}$ by construction. The last claim follows from the fact that $\rho$ is injective.
\end{proof}
The classical Buchberger's Algorithm relies on a result known as Buchberger's Criterion.  It determines when a set forms a Gr\"obner basis by analyzing so-called \emph{S-polynomials}. We give the corresponding construction in the OI-setting below.

\begin{defn}
Let $ae_{\pi,i}$ and $be_{\rho,j}$ be monomials in $\Fb$. Define
\[
\lcm(ae_{\pi,i},be_{\rho,j})=\begin{cases}
\lcm(a,b)e_{\pi,i}&\text{if }e_{\pi,i}=e_{\rho,j}\\
0&\text{otherwise}.
\end{cases}
\]
Here, $\lcm(a,b)$ is the least common multiple of the monomials $a$ and $b$ in $\Pb$.
\end{defn}
\begin{defn}
Let $f,g\in\Fb_m$ be nonzero. The \emph{S-polynomial} of $f$ and $g$ is the combination
\[
S(f,g)=\frac{\lcm(\lm(f),\lm(g))}{\lt(f)}f-\frac{\lcm(\lm(f),\lm(g))}{\lt(g)}g.
\]
\end{defn}

\begin{rem}
\label{remark:cancellation}
By design, S-polynomials produce cancellation of leading terms. Specifically, one has $\lm(S(f,g))<\lcm(\lm(f),\lm(g))$ for any nonzero $f,g\in\Fb_m$.
\end{rem}

We need the following observation  whose straightforward proof is left to the reader.

\begin{lem}
\label{cor:equivariantspoly}
For any monomials $\mu,\nu\in\Fb_m$, any nonzero $f,g\in\Fb_m$ and any $\rho\in\hom(m,n)$, one has 
\[
\Pb(\rho) \big (\frac{\lcm(\mu,\nu)}{\mu}\big )=\frac{\lcm(\rho_*(\mu),\rho_* (\nu))}{\rho_* (\mu)}.
\]
and
\[
\rho_*(S(f,g))=S(\rho_*(f), \rho_*(g)), 
\]
where $\rho_* = \Fb(\rho)$. 
\end{lem}

Central to our theory is the notion of a \emph{critical pair}, which we define as follows. 

\begin{defn}
Let $B\subseteq\Fb$. A tuple
\[
(\Fb(\sigma)(b_i),\Fb(\tau)(b_j))\in\orb(B,m)\times\orb(B,m)
\]
is called a \emph{critical pair of $B$} if $b_i, b_j \neq 0$ and  $\lm(b_i)$ and $\lm(b_j)$ involve the same basis element of $\Fb$ and $m=|\im(\sigma)\cup\im(\tau)|$. The set of all critical pairs of $B$ is denoted $\Cc(B)$.
\end{defn}

An important property of critical pairs is the fact that if $B$ is finite, then $\Cc(B)$ is finite. Indeed, using the above notation,  finiteness of $B$ implies that there are only finitely many choices for $b_i, b_j \in B$, and for each such pair there are only finitely many possible $m$ because $m=|\im(\sigma)\cup\im(\tau)|\leq w(b_i)+w(b_j)$. The sets $\hom (w(b_i), m)$ and $\hom (w(b_j), m)$ are finite. 
 
We are now ready to state and prove an OI-version of Buchberger's Criterion.   

\begin{thm}[OI-Buchberger's Criterion]
\label{theorem:buchbergercriterion}
Let $\Mb=\la B\ra_{\Fb}$ be the \OI-module generated by a subset $B\subseteq\Fb$. 
Then $B$ is a Gr\"obner basis of $\Mb$ if and only if $S(f,g)$ has a remainder of zero modulo $B$ for every critical pair $(f,g)\in\Cc(B)$.
\end{thm}

\begin{proof}
Suppose first that $B$ is a Gr\"obner basis of $\Mb$. Assume there is some $(f,g)\in\Cc(B)$ such that $S(f,g)$ has a nonzero remainder $r$ modulo $B$. Then $S(f,g)=q+r$ for some $q\in\Mb$.  Since $f,g\in\Mb$ we have $S(f,g)\in\Mb$. It follows that $r\in\Mb$, and so  $\lm(r) \in \la \lm(\Mb)\ra_{\Fb}=\la\lm(B)\ra_{\Fb}$. Thus,  $\lm(r)$ is OI-divisible by an element of $\lm(B)$, a contradiction to the fact that $r$ is a remainder modulo $B$.

Conversely, suppose each $S(f,g)$ with $(f,g)\in\Cc(B)$ has a remainder of zero modulo $B$. By Lemma \ref{lemma:widthwisegroebner} it suffices to show  for each $n\in\Z_{\geq0}$ that $B_n =\orb(B,n)$ is a Gr\"obner basis of $\Mb_n$ with respect to $<_n$. 
Pick $f,g\in B_n$ and write $f=\Fb(\sigma)(b_1)$ and $g=\Fb(\tau)(b_2)$ for elements $b_1,b_2\in B$ and maps $\sigma\in\hom(w(b_1),n)$ and $\tau\in\hom(w(b_2),n)$. Using the classical Buchberger's Criterion (see, e.g., \cite[Theorem 15.8]{Eis95}), we need to show that $S(f,g)$  has a remainder of zero modulo $B_n$.

Indeed, if $\lm(b_1)$ and $\lm(b_2)$ involve different basis elements of $\Fb$ then $S(f,g)$ is zero, and we are done. Assume $\lm(b_1)$ and $\lm(b_2)$ involve the same basis element of $\Fb$.  
By Lemma \ref{lemma:oifactorization}, there are maps $\bar\sigma\in\hom(w(b_1),\l)$, $\bar\tau\in\hom(w(b_2),\l)$ and $\rho\in\hom(\l,n)$ with $\l=|\im(\sigma)\cup\im(\tau)| = |\im(\bar \sigma)\cup \im(\bar \tau)|$ such that $\sigma=\rho\circ\bar{\sigma}$ and $\tau=\rho\circ\bar\tau$. 
Put \[\bar f=\Fb(\bar\sigma)(b_1)\quad\text{and}\quad\bar g=\Fb(\bar\tau)(b_2)\] and note that $(\bar f,\bar g)\in\Cc(B)$. Functoriality of $\Fb$ gives $\Fb(\rho)(\bar f)=f$ and $\Fb(\rho)(\bar g)=g$. 
By \Cref{cor:equivariantspoly}, we get 
\[
\Fb(\rho)(S(\bar f, \bar g))=S(\Fb(\rho)(\bar f),\Fb(\rho)(\bar g))=S(f,g). 
\]
Since $S(\bar f,\bar g)$ has a remainder of zero modulo $B$ by assumption, it follows that $S(f,g)$ also has a remainder of zero modulo $B$. So we can write $S(f,g)=\sum a_iq_i$ for some $a_i\in\Pb_n$ and some $q_i\in\orb(B,n)$ such that $\lm(a_iq_i)\leq\lm(S(f,g))$ whenever $a_iq_i\not=0$. As $<$ and $<_n$ agree on $\Fb_n$, this says that $S(f,g) \in \Mb_n$ has a remainder of zero modulo $B_n$ with respect to the  order $<_n$, completing the argument. 
\end{proof}

The OI-Buchberger's Criterion leads to an algorithm for computing Gr\"obner bases in finite time. The idea is analogous to the classical case: check if any S-polynomial has nonzero remainder, and if so, append the remainder and repeat

\begin{algorithm}[OI-Buchberger's Algorithm]\label{algorithm:buchberger}
\begin{algorithmic}[1]
\Require A finite subset $B\subset\Fb$
\Ensure A finite Gr\"obner basis $G$ of $\la B\ra_{\Fb}$
\State set $G=B$
\State set $P=\mcal{C}(G)$
\While{$P\not=\emptyset$}
\State choose $(f,g)\in P$
\State set $P= P\;\backslash\;\{(f,g)\}$
\State let $r$ be a remainder of $S(f,g)$ modulo $G$
\If{$r\not=0$}
\State set $G= G\cup\{r\}$
\State set $P=\mcal{C}(G)$
\EndIf
\EndWhile
\end{algorithmic}
\end{algorithm}



\begin{proof}[Correctness and Termination of \ref{algorithm:buchberger}]
Let $G_1,G_2,\ldots$ denote the values of $G$ each time it gets updated (so in particular, $G_1=B$). Suppose the algorithm terminates and returns $G_t$. Then $S(f,g)$ has a remainder of zero modulo $G_t$ for every $(f,g)\in\mcal{C}(G_t)$, so $G_t$ forms a finite Gröbner basis for $\la G_t\ra_{\Fb}$ by the OI-Buchberger's Criterion. For any $t$, if $(f,g)\in\mcal{C}(G_t)$ and $S(f,g)$ has a remainder of $r\not=0$ modulo $G_t$, then $r\in\la G_t\ra_{\Fb}$, and so $\la G_t\ra_{\Fb}=\la G_{t+1}\ra_{\Fb}$. It follows that $\la G_t\ra_{\Fb}=\la B\ra_{\Fb}$ and the algorithm is correct. \par Suppose now that the algorithm does not terminate. Since each $\mcal{C}(G_n)$ is finite we obtain an infinite sequence $G_1\subseteq G_2\subseteq\cdots$. Let $n\in\N$ and let $r\not=0$ be the remainder modulo $G_n$ used to form $G_{n+1}$. By the definition of remainder, we have $\lm(r)\not\in\la\lm(G_n)\ra_{\Fb}$, and hence $\la\lm(G_n)\ra_{\Fb}\subsetneq\la\lm(G_{n+1})\ra_{\Fb}$. This yields a strictly increasing chain of submodules of $\Fb$, contradicting the fact that $\Fb$ is Noetherian (see \cite[Proposition 4.2 and Theorem 6.15]{NR19}). Hence the algorithm terminates in finite time.
\end{proof}

The main result of this section follows easily now. 

\begin{thm}
   \label{thm:finite compute G}
Every submodule of a finitely generated free \OI-module $\Fb$ over 
$\Pb \cong (\Xb^{\OI,1})^{\otimes c}$  has a finite Gr\"obner basis with respect to any monomial order on $\Fb$ that can be computed in finitely many steps. 
\end{thm}

\begin{proof}
By \cite[Theorem 6.15]{NR19}, any submodule of $\Fb$ is finitely generated. Thus, we conclude by applying \Cref{algorithm:buchberger}.  
\end{proof}

The first author has implemented \Cref{algorithm:buchberger} using Macaulay2 (see \cite{M}). 

A finite Gr\"obner basis $B$ of an OI-module $\Mb$ is said to be \emph{minimal} if every proper subset of $B$ is no longer a Gr\"obner basis $B$. 

\begin{ex}
Let $\Fb=\Fb^{\OI,1}\oplus\Fb^{\OI,1}\oplus\Fb^{\OI,2}$ have basis $\{e_{\id_{[1]},1},e_{\id_{[1]},2},e_{\id_{[2]},3}\}$. Suppose $c=2$ so $\Pb$ has two rows of variables, and let
\[
B=\{x_{1,1}e_{\id_{[1]},1}+x_{2,1}e_{\id_{[1]},2},\;x_{1,2}x_{1,1}e_{\pi,2}+x_{2,2}x_{2,1}e_{\id_{[2]},3}\}
\]
where $\pi:[1]\to[2]$ is given by $1\mapsto 2$.  So, the elements of $B$ have width one and two, respectively. 
Using the Macaulay2 code of \cite{M}, one obtains  a minimal Gr\"obner basis for $\la B\ra_{\Fb}$ with respect to the lex order from Example \ref{example:lex}: 
\begin{align*}
x_{1,1}e_{\id_{[1]},1}+x_{2,1}e_{\id_{[1]},2}&\in\Fb_1\\
x_{1,2}x_{1,1}e_{\pi,2}+x_{2,2}x_{2,1}e_{\id_{[2]},3}&\in\Fb_2\\
-x_{2,3}x_{2,2}x_{1,1}e_{\sigma_1,3}+x_{2,3}x_{2,1}x_{1,2}e_{\sigma_2,3}&\in\Fb_3
\end{align*}
where $\sigma_1:[2]\to[3]$ is given by $1\mapsto 2$ and $2\mapsto 3$ and $\sigma_2:[2]\to[3]$ is given by $1\mapsto 1$ and $2\mapsto 3$.  In particular,  the OI-Buchberger's Algorithm needed to append exactly one more element of width $3$ before terminating.
\end{ex}

Given a submodule $\Mb$ of $\Fb$, the fact that $\Mb$ has a finite Gr\"obner basis $G$ combined with \Cref{lemma:widthwisegroebner} shows in particular that a Gr\"obner basis of $\Mb$ can be obtained as the union of Gr\"obner bases of $\Mb_n$ with respect to the order $<_n$ with $n \le w$, where $w$ is the maximum width of an element in $G$. This suggests an alternative way for computing a Gr\"obner basis of $\Mb$ by determining and comparing Gr\"obner bases of $\Mb_n$ with respect to the order $<_n$ for sufficiently large $n$. The problem is that a priori one does not know when to stop, i.e., what is large enough. The following result gives a sufficient condition for the desired stabilization. 

\begin{prop}
   \label{prop:stabilization} 
Let $\Mb$ be a submodule of a free $\Pb$-module $\Fb$ generated by a finite set $B \subset \Fb$. Let $<$ be any monomial order on $\Fb$. For any integer $m \ge 0$ , denote by $<_m$ the restriction of $<$ to $\Fb_m$. 
Let $W \ge 1$ denote the maximum width of an element  in $B$. Then $B$ is a Gr\"obner basis of $\Mb$ with respect to $<$ if and only if $\orb(B,m)$ is a Gr\"obner basis for $\Mb_m$ with respect to $<_m$ for each integer $m \le 2W$.
\end{prop}

\begin{proof}
The forward direction is immediate from Lemma 3.6. For the reverse, by the OI-Buchberger's Criterion, it suffices  to show  for any critical pair 
$(f,g) \in \Cc(B)$ that $S(f,g)$ has a remainder of zero modulo $B$. Considering such a critical pair, we can write $(f,g)=(\Fb(\sigma)(b_1),\Fb(\tau)(b_2))$ for maps $\sigma$ and $\tau$ and elements $b_1,b_2\in B$ such that $\sigma,\tau$ land in width $m=|\im(\sigma)\cup\im(\tau)|\leq w(b_1)+w(b_2)\leq 2W$. 
Hence our hypothesis gives that $\orb(B,m)$ is a Gr\"obner basis for $\Mb_m$,  and so  $S(f,g)$ has a remainder of zero modulo $\orb(B,m)$ with respect to $<_m$. But since $<$ and $<_m$ agree on $\Mb_m$, this says that $S(f,g)$ has a remainder of zero modulo $B$ with respect to $<$. 
\end{proof}


\section{Syzygies and the OI-Schreyer's Theorem}
\label{section:schreyer}
Recall that in the classical setting, if $M$ is a finitely generated submodule of a free module of finite rank over a Noetherian polynomial ring, Schreyer's Theorem \cite[Theorem 15.10]{Eis95} computes a finite Gr\"obner basis for the module of syzygies of $M$. Using the theory developed in the previous section, we extend Schreyer's result to the setting of OI-modules. 

We continue to use the above notation. In particular,  $\Fb=\bigoplus_{i=1}^s\Fb^{\OI,d_i}_{\Pb}$  always denotes a finitely generated, free OI-module over $\Pb=(\Xb^{\OI,1})^{\otimes c}$ with some $c \ge 1$.  We fix a monomial order $<$ on $\Fb$. 

Our computation of syzygies of a submodule $\Mb$ of $\Fb$ depends on a choice of a generating set of $\Mb$.

\begin{defn} 
   \label{def:phi}
Let $B=\{b_1,\ldots,b_t\}\subset\Fb$ and let $\Gb=\bigoplus_{i=1}^t\Fb^{\OI,w(b_i)}_{\Pb}$ be a free OI-module with basis $\{\varepsilon_{\id_{[w(b_i)]},i} \ : \ i \in [t]\}$. We define the \emph{module of syzygies of $B$} as $\syz(B)=\ker(\phi)$, where $\phi$ denotes the surjective morphism 
\[
\phi \colon \Gb = \bigoplus_{i=1}^t \langle  \varepsilon_{\id_{[w(b_i)]},i} \rangle_{\Fb}   \to\la B\ra_{\Fb}, \text{ defined by $\varepsilon_{\id_{[w(b_i)]},i}\mapsto b_i$}. 
\]
\end{defn}

Our goal in this section is to compute a finite Gr\"obner basis for $\syz(B)$. To begin, we will need a suitable monomial order on $\Gb$. 

\begin{defn}
Let $B=\{b_1,\ldots,b_t\}\subset\Fb$.
\begin{enumerate}
\item For any $i \in [t]$, order the set $\{  \hom(w(b_i), n) \, : \, n \ge w(b_i) \}$:  Given $\pi\in\hom(w(b_i),m)$ and $\rho\in\hom(w(b_i),n)$, define $\pi<\rho$ if 
 $(m,\pi(1),\ldots,\pi(w(b_i)))<(n,\rho(1),\ldots,\rho(w(b_i)))$ in the lexicographic order on $\N^{w(b_i)+1}$, where $\N$ is ordered by $1 < 2 < \cdots$.


\item  Let $\prec_B$ be the total order on the union of the above sets
\[
\{(\pi,i)\;:\;i\in[t],\;\pi\in\hom(w(b_i),m),\; m\geq w(b_i)\}
\]
defined by $(\pi,i)\prec_B(\rho,j)$ if either $i<j$ or $i=j$ and $\pi<\rho$.

\item Define a total order $<_B$ on $\mon(\Gb)$ by setting $a\varepsilon_{\pi,i}<_Bb\varepsilon_{\rho,j}$ if either $ \lm(\phi(a\varepsilon_{\pi,i}))<\lm(\phi(b\varepsilon_{\rho,j}))$ or $\lm(\phi(a\varepsilon_{\pi,i}))=\lm(\phi(b\varepsilon_{\rho,j}))$ and $(\rho,j)\prec_B(\pi,i)$.  

\end{enumerate}
\end{defn}

\begin{rem}
Note that for any monomial $a\varepsilon_{\pi,i}\in\mon(\Gb)$, one has
\[
\lm(\phi(a\varepsilon_{\pi,i}))=\lm(a\Fb(\pi)(b_i))=a\Fb(\pi)(\lm(b_i)) = a \pi_* (\lm (b_i)).
\]
Thus, Part (iii) can be  more explicitly re-stated  as  $a\varepsilon_{\pi,i}<_Bb\varepsilon_{\rho,j}$  if either $a \pi_* (\lm (b_i)) < b \rho_* (\lm (b_j))$ or $a \pi_* (\lm (b_i)) = b \rho_* (\lm (b_j))$ and $(\rho,j)\prec_B(\pi,i)$.
\end{rem}

 Together with the fact that $<$ is a monomial order on $\Fb$, the following observation implies that $<_B$ is a monomial order on $\Gb$.

\begin{lem}
\label{lemma:equivariant-indices}
Consider any maps   $\pi\in\hom(w(b_i),m)$ and $\rho\in\hom(w(b_j),m)$ with $i,j\in[t]$ and 
$\sigma\in\hom(m,n)$. Then $(\sigma\circ\pi,i)\prec_B(\sigma\circ\rho,j)$ if and only if $(\pi,i)\prec_B(\rho,j)$.
\end{lem}
\begin{proof}
This follows from the fact that $\sigma$ is a strictly increasing map.
\end{proof}

For the remainder of this section we assume that $B=\{b_1,\ldots,b_t\}$ is a Gr\"obner basis of $\la B\ra_{\Fb}$ and that each $b_i$ is monic, possibly after multiplying by a suitable element of $K$.  We need some further notation. 

For any $i,j\in[t]$, $\sigma\in\hom(w(b_i),m)$ and $\tau\in\hom(w(b_j),m)$ with $m\geq\max(w(b_i),w(b_j))$,  the remainder of $S(\Fb(\sigma)(b_i),\Fb(\tau)(b_j))$ modulo $B$ is zero since $B$ is a Gr\"obner basis. Hence the division algorithm gives an expression 
\[
S(\Fb(\sigma)(b_i),\Fb(\tau)(b_j))=\sum_\l a_{i,j,\l}^{\sigma,\tau}\Fb(\pi_{i,j,\l}^{\sigma,\tau})(b_{k_{i,j,\l}^{\sigma,\tau}})
\]
with $a_{i,j,\l}^{\sigma,\tau}\in\Pb_m$ and $\Fb(\pi_{i,j,\l}^{\sigma,\tau})(b_{k_{i,j,\l}^{\sigma,\tau}})\in\orb(B,m)$, where
\begin{equation}
\label{equation:lm-inequality}
\lm(u\Fb(\pi_{i,j,\l}^{\sigma,\tau})(b_{k_{i,j,\l}^{\sigma,\tau}}))\leq\lm(S(\Fb(\sigma)(b_i),\Fb(\tau)(b_j)))
\end{equation}
for any monomial $u$ in $a_{i,j,\l}^{\sigma,\tau}$ (see Remark \ref{rem:r3}). Define
\begin{equation}
\label{eq:syz}
s_{i,j}^{\sigma,\tau}=m_{i,j}^{\sigma,\tau}\varepsilon_{\sigma,i}-m_{j,i}^{\tau,\sigma}\varepsilon_{\tau,j}-\sum_\l a_{i,j,\l}^{\sigma,\tau}\varepsilon_{\pi_{i,j,\l}^{\sigma,\tau},k_{i,j,\l}^{\sigma,\tau}}\in\Gb_m
\end{equation}
where
\[
m_{i,j}^{\sigma,\tau}=\frac{\lcm(\Fb(\sigma)(\lm(b_i)),\Fb(\tau)(\lm(b_j)))}{\Fb(\sigma)(\lm(b_i))}\in\Pb_m.
\]

We need one more preparatory observation. 

\begin{lem}
\label{lemma:lm-syzygy}
If $(\sigma,i)\prec_B(\tau,j)$ then $\lm_{<_B}(s_{i,j}^{\sigma,\tau})=m_{i,j}^{\sigma,\tau}\varepsilon_{\sigma,i}$.
\end{lem}

\begin{proof}
Note that $\phi (m_{i,j}^{\sigma,\tau}\varepsilon_{\sigma,i}) = 
\frac{\lcm (\sigma_* (\lm (b_i)), \tau_* (\lm (b_j)))}{\sigma_* (\lm (b_i))} \sigma_* (b_i)$, and so its leading 
monomial is  $\lcm (\sigma_* (\lm (b_i)), \tau_* (\lm (b_j)))$. Analogously, we see that $\phi(m_{j,i}^{\tau,\sigma}\varepsilon_{\tau,j})$ has the same leading monomial. Hence  $(\sigma,i)\prec_B(\tau,j)$ implies $m_{j,i}^{\tau,\sigma}\varepsilon_{\tau,j}<_B m_{i,j}^{\sigma,\tau}\varepsilon_{\sigma,i}$ by the definition of $<_B$. Now fix some $\l$ and let $u$ be the  monomial in $a_{i,j,\l}^{\sigma,\tau}$ such that 
\[
\lm_{<_B}(a_{i,j,\l}^{\sigma,\tau}\varepsilon_{\pi_{i,j,\l}^{\sigma,\tau},k_{i,j,\l}^{\sigma,\tau}})=u\varepsilon_{\pi_{i,j,\l}^{\sigma,\tau},k_{i,j,\l}^{\sigma,\tau}}. 
\]
Then we obtain
\begin{align*}
\lm(u\Fb(\pi_{i,j,\l}^{\sigma,\tau})(b_{k_{i,j,\l}^{\sigma,\tau}}))
&\leq\lm(S(\Fb(\sigma)(b_i),\Fb(\tau)(b_j)))\tag{by \eqref{equation:lm-inequality}}\\
&<\lcm(\Fb(\sigma)(\lm(b_i)),\Fb(\tau)(\lm(b_j)))\tag{by Remark \ref{remark:cancellation}}\\
&=m_{i,j}^{\sigma,\tau}\Fb(\sigma)(\lm(b_i)). 
\end{align*}
By definition of $<_B$, this gives $u\varepsilon_{\pi_{i,j,\l}^{\sigma,\tau},k_{i,j,\l}^{\sigma,\tau}}<_B m_{i,j}^{\sigma,\tau}\varepsilon_{\sigma,i}$ , and the claim follows.
\end{proof}

We are now ready to state and prove the main result of this section.  

\begin{thm}[OI-Schreyer's Theorem]
\label{theorem:oi-schreyers-theorem}
The $s_{i,j}^{\sigma,\tau}$ with $(\Fb(\sigma)(b_i),\Fb(\tau)(b_j))\in\Cc(B)$ and $(\sigma,i)\prec_B(\tau,j)$ form a finite Gr\"obner basis for $\syz(B)$ with respect to $<_B$.
\end{thm}
\begin{proof}
First note that each $s_{i,j}^{\sigma,\tau}$ is indeed a syzygy, for 
\[
\phi(s_{i,j}^{\sigma,\tau})=S(\Fb(\sigma)(b_i),\Fb(\tau)(b_j))-S(\Fb(\sigma)(b_i),\Fb(\tau)(b_j))=0.
\]
Now let $h\in\syz(B)$ be any nonzero syzygy. We need to show that $\lm (h)$ is divisible by the leading monomial of one of the $s_{i,j}^{\sigma,\tau}$ in the claimed Gr\"obner basis of $\syz (B)$. 

To see this write $h=\sum_{\l=1}^n c_\l x^{p_\l}\varepsilon_{\pi_\l,i_\l}$ for some nonzero $c_i\in K$ and distinct $x^{p_\l}\varepsilon_{\pi_\l,i_\l}\in\mon(\Gb)$. Note that $n\geq 2$ since $h$ is a nonzero syzygy. Assume the monomials of $h$ are ordered so that $x^{p_\l}\varepsilon_{\pi_\l,i_\l}>_B x^{p_{\l'}}\varepsilon_{\pi_{\l'},i_{\l'}}$ if $\l<\l'$. Then $\lt_{<_B}(h)=c_1x^{p_1}\varepsilon_{\pi_1,i_1}$ and, by  definition of $<_B$, we get
\begin{equation}
\label{eq:lmineq}
x^{p_1}\Fb(\pi_1)(\lm(b_{i_1}))\geq \cdots\geq x^{p_n}\Fb(\pi_n)(\lm(b_{i_n})).
\end{equation}
We claim that
\begin{equation}
\label{eq:lmeq}
x^{p_1}\Fb(\pi_1)(\lm(b_{i_1}))=x^{p_2}\Fb(\pi_2)(\lm(b_{i_2})).
\end{equation}
Indeed,  otherwise \eqref{eq:lmineq} becomes
\begin{equation}
\label{eq:newlmineq}
x^{p_1}  \Fb(\pi_1)(\lm(b_{i_1}))>x^{p_2}\Fb(\pi_2)(\lm(b_{i_2}))\geq \cdots\geq x^{p_n}\Fb(\pi_n)(\lm(b_{i_n})).
\end{equation}
Since $h$ is a syzygy we have $\sum_{\l=1}^n c_\l x^{p_\l}\Fb(\pi_\l)(b_{i_\l})=0$ which, using \eqref{eq:newlmineq}, implies that $x^{p_1}\Fb(\pi_1)(\lm(b_{i_1}))$ can be written as a $K$-linear combination of strictly smaller monomials. This is impossible, and hence we have established \eqref{eq:lmeq}. 

Since $x^{p_1} \varepsilon_{\pi_1,i_1} >_B x^{p_2} \varepsilon_{\pi_2,i_2}$
and these two elements are mapped by $\phi$ onto the  same element of $\Fb$ by \eqref{eq:lmeq}, the definition of $<_B$ implies 
\begin{equation}
   \label{eqn:lead}
(\pi_1,i_1)\prec_B(\pi_2,i_2). 
\end{equation}
Now set $\alpha=x^{p_1}\Fb(\pi_1)(\lm(b_{i_1}))$. 
By Equation \eqref{eq:lmeq}, $\alpha$ is divisible by $\Fb(\pi_2)(\lm(b_{i_2}))$, and so $\alpha$ is divisible by 
$\beta =\lcm(\Fb(\pi_1)(\lm(b_{i_1})),\Fb(\pi_2)(\lm(b_{i_2})))$. By the OI-Factorization Lemma,  there are maps $\overline{\pi_1}\in\hom(w(b_{i_1}), \l)$, $\overline{\pi_2}\in\hom(w(b_{i_2}), \l)$ and $\rho\in\hom(\l,m)$ with $\l=|\im(\overline{\pi_1})\cup\im(\overline{\pi_2})|$ such that
\[
\rho\circ\overline{\pi_1}=\pi_1\quad\text{and}\quad\rho\circ\overline{\pi_2}=\pi_2.
\]
By construction,  $(\Fb(\overline{\pi_1})(b_{i_1}),\Fb(\overline{\pi_2})(b_{i_2}))$ is in $\mcal{C}(B)$.  Moreover, Relation \eqref{eqn:lead}  and Lemma \ref{lemma:equivariant-indices} imply 
$(\overline{\pi_1}, i_1)\prec_B(\overline{\pi_2},i_2)$. Hence, we get 
\begin{align*}
\frac{\alpha}{\beta}\Gb(\rho)(\lm_{<_B}(s_{i_1,i_2}^{\overline{\pi_1},\overline{\pi_2}}))&=\frac{\alpha}{\beta}\Gb(\rho)(m_{i_1,i_2}^{\overline{\pi_1},\overline{\pi_2}}\varepsilon_{\overline{\pi_1},i_1})\tag{by Lemma \ref{lemma:lm-syzygy}}\\
&=\frac{\alpha}{\beta}\frac{\lcm(\Fb(\pi_1)(\lm(b_{i_1})), \Fb(\pi_2)(\lm(b_{i_2})))}{\Fb(\pi_1)(\lm(b_{i_1}))}\varepsilon_{\pi_1,i_1}\tag{by  \Cref{cor:equivariantspoly}}\\
&=\frac{x^{p_1}\Fb(\pi_1)(\lm(b_{i_1}))}{\Fb(\pi_1)(\lm(b_{i_1}))}\varepsilon_{\pi_1,i_1}\\
&=x^{p_1}\varepsilon_{\pi_1,i_1}\\
&=\lm_{<_B}(h).
\end{align*}
This shows that $\lm_{<_B}(h)$ is OI-divisible by $\lm_{<_B}(s_{i_1,i_2}^{\overline{\pi_1},\overline{\pi_2}})$,  which completes the proof.
\end{proof}

In the graded case, the method can be modified to give homogenous syzygies. 

\begin{rem}
   \label{rem:graded case} 
If $\Pb$ and $\Fb$ are graded and the elements of $B$ are homogenous, we use the morphism   
\[
\phi \colon \Gb = \bigoplus_{i=1}^t \Fb^{\OI,w(b_i)}(-\deg(b_i))   \to\la B\ra_{\Fb}, \text{ defined by $\varepsilon_{\id_{[w(b_i)]},i}\mapsto b_i$}. 
\]
It is graded (of degree zero). One checks that as a consequence the syzygies defined in \eqref{eq:syz} are homogeneous.  
\end{rem}

If the generating set $B$ consists of monomials, then one gets a more explicit Gr\"obner basis of $\syz (B)$. The result is an OI-version of the description of the syzygy module of a monomial module in the classical situation (see, e.g, \cite[Lemma 15.1]{Eis95}). 

\begin{cor}

For any finite set of monomials $B \subset \Fb_n$, the elements $s_{i,j}^{\sigma,\tau} \in \Gb$ with $(\Fb(\sigma)(b_i),\Fb(\tau)(b_j))\in\Cc(B)$ and $(\sigma,i)\prec_B(\tau,j)$ form a finite Gr\"obner basis for $\syz(B)$ with respect to $<_B$, where 
\[
s_{i,j}^{\sigma,\tau}=\frac{\lcm(\Fb(\sigma)(b_i),\Fb(\tau)(b_j))}{\Fb(\sigma)(b_i)}\varepsilon_{\sigma,i}-\frac{\lcm(\Fb(\sigma)(b_i),\Fb(\tau)(b_j))}{\Fb(\tau)(b_j)}\varepsilon_{\tau,j}. 
\]
\end{cor}

\begin{proof}
Since $B$ consists of monomials, it is a Gr\"obner basis of $\langle B \rangle_{\Fb}$ with respect to any monomial order on $\Fb$. Thus, \Cref{theorem:oi-schreyers-theorem} is applicable. Moreover, each S-polynomial   $S(\Fb(\sigma)(b_i),\Fb(\tau)(b_j))$ is zero, which implies the stated form of each $s_{i,j}^{\sigma,\tau}$ (see Equation \eqref{eq:syz}).  We conclude by \Cref{theorem:oi-schreyers-theorem}. 
\end{proof}


\section{Computing Free Resolutions}
\label{section:resolution}
We now apply our results to the problem of computing free resolutions. We continue to use the above notation. 

Recall that if $\Mb$ is an $\Ab$-module, then a \emph{free resolution} of $\Mb$ is an exact sequence
\begin{equation*}
   \label{eq:free res}
\Fb^\bullet: \quad \cdots\to\Fb^2\to\Fb^1\to\Fb^0\to\Mb\to0, 
\end{equation*}
where each $\Fb^i$ is a free $\Ab$-module. If $\Mb$ is a finitely generated $\Pb$-module then  $\Mb$ admits a free resolution in which every free module $\Fb^i$ is finitely generated (see \cite[Theorem 7.1]{NR19}). The argument in \cite{NR19} is based on the finiteness results established in that paper and not constructive. Using the results of the previous section we obtain the first algorithm to compute, for 
any integer $p \ge 0$, the first $p$ steps of a free resolution of a submodule $\Mb$ of $\Mb$. Indeed, given a finite Gr\"obner basis  $B_0$ of $\Mb$, define $\phi_0 \colon \Fb^0\to\Mb$ by mapping the basis elements of $\Fb^0$ onto the corresponding generators in $B_0$. Now let $B_1$ be a generating set of $\ker(\phi_0)$ determined by \Cref{theorem:oi-schreyers-theorem}  and repeat this process.


\begin{proc}[Method for Computing Free Resolutions]
\label{proc:freeres}
Let $\Mb$ be a submodule of a free $\Pb$-module. Suppose $\Mb$ is finitely generated by a subset $B$. Compute a free resolution for $\Mb$ as follows:
\begin{enumerate}
\item Determine a finite Gr\"obner basis $B_0$ for $\Mb$ using the OI-Buchberger's Algorithm applied to the set $B$.
\item Apply the OI-Schreyer's Theorem to $B_0$ to obtain a finite Gr\"obner basis $B_1$ of $\syz(B_0)$.
\item Repeat (ii) as many times as desired to compute a finite Gr\"obner basis $B_n$ of $\syz(B_{n-1})$.
\item Form (the beginning of) the free resolution $\Fb^\bullet$ by defining $\Fb^i$ such  that $\phi_i$ maps the basis elements of $\Fb^i$ onto the corresponding generators in $B_i$ as in \Cref{def:phi}. 
\end{enumerate}
\end{proc}

Using \Cref{rem:graded case}, the above procedure can be modified in the case that $\Mb$ is a graded OI-module to produce a graded free resolution of $\Mb$. 
%

If $\Mb$ is a graded submodule of $\Fb$, then it is shown in \cite[Theorem 3.10]{FN21} that $\Mb$ has a graded free resolution $\Fb^\bullet$ such that for any other graded free resolution $\Gb^\bullet$ and for any homological degree $d$, the rank of $\Fb^d$ is at most the rank of $\Gb^d$. By \cite[Theorem 3.10]{FN21}, the resolution 
$\Fb^\bullet$ is unique up to isomorphisms of graded free resolutions and called a 
\emph{minimal} graded free resolution of $\Mb$ in \cite{FN21}. Our methods allow us to show that there is an algorithm that produces such a minimal free resolution. 

As in the classical situation, a key is to characterize the maps occurring in a graded minimal free resolution. 

\begin{defn}[{\cite[Definition 3.1]{FN21}}]
Let $\Fb$ and $\Gb$ be free $\Pb$-modules with bases $\{f_1,\ldots,f_s\}$ and $\{g_1,\ldots,g_t\}$ where each $f_i$ lives in width $u_i$ and each $g_j$ lives in width $v_j$. A morphism $\phi \colon \Fb\to\Gb$ is determined by the images of the $f_i$, i.e.\  by $s$ expressions of the form
\[
\phi(f_i)=\sum_{\substack{1\leq j\leq t\\\epsilon_{i,j}\in\hom(v_j,u_i)}}a_{\epsilon_{i,j}}\Gb(\epsilon_{i,j})(g_j)
\]
where each $a_{\epsilon_{i,j}}\in\Pb_{u_i}$. We say $\phi$ is \emph{minimal} if whenever any $\epsilon_{i,j}$ is an identity map, the coefficient $a_{\epsilon_{i,j}}$ is not a unit. 
\end{defn}

This concept allows one to check whether a graded free resolution is minimal. Indeed, \cite[Theorem 3.10]{FN21} gives the following characterization:

\begin{thm}
A graded free resolution $\Fb^\bullet$ of $\Mb$ is minimal if and only if each map between free modules in $\Fb^\bullet$ is minimal.
\end{thm}

We are ready to establish that 
the first steps of a graded minimal free resolution can be computed algorithmically. 

\begin{thm}
   \label{thm:MFR}
If $\Mb$ is a graded submodule of $\Fb$ with a finite generating set $B$ consisting of homogenous elements, then, for any integer $p \ge 0$, there is a finite algorithm that determines the first $p$ steps 
\[
\Fb^p \to \cdots\to\Fb^2\to\Fb^1\to\Fb^0\to\Mb\to0  
\]
 in a graded minimal free resolution of $\Mb$. 
\end{thm}

\begin{proof}
Using the graded version of Procedure \ref{proc:freeres}, we compute a graded exact sequence 
\[
\Fb^p \to \cdots\to\Fb^2\to\Fb^1\to\Fb^0\to\Mb\to0, 
\]
where each $\Fb^i$ is a finitely generated graded free $\Pb$-module. If one of the maps in this sequence is not minimal we prune the sequence to obtain an exact sequence such that two consecutive free modules are replaced by free modules of strictly smaller rank. Repeating this process as often as needed beginning with the right-most map one eventually obtains (the beginning) of a graded minimal free resolution of $\Mb$. 

The pruning is described in detail in the proof of  \cite[Lemma 3.5]{FN21}. 
\end{proof}

\Cref{proc:freeres} and the algorithm in \Cref{thm:MFR} have been implemented in \cite{M}.

\begin{ex}\mbox{}
\begin{enumerate}
\item Let $\Fb=\Fb^{\OI,1}\oplus\Fb^{\OI,2}$ have basis $\{e_{\id_{[1]},1},e_{\id_{[2]},2}\}$. Let $c=2$, that is, $\Pb$ has two rows of variables, and consider the submodule $\Mb$ of $\Fb$ generated by the single element $f=x_{1,2}x_{1,1}e_{\pi,1}+x_{2,2}x_{2,1}e_{\rho,2}\in\Fb_3$ where $\pi:[1]\to[3]$ is given by $1\mapsto 2$ and $\rho:[2]\to[3]$ is given by $1\mapsto 1$ and $2\mapsto3$.\par The first step in resolving $\Mb$ is to compute a Gr\"obner basis for $\Mb$ with respect to the lex order. Unlike in the classical setting, the one-element set $\{f\}$ does \emph{not} form a Gr\"obner basis for $\Mb$. In fact a minimal
Gr\"obner basis for $\Mb$ can be computed using the Macaulay2 script in \cite{M}:
\begin{align*}
x_{1,2}x_{1,1}e_{\pi,1}+x_{2,2}x_{2,1}e_{\rho,2}&\in\Fb_3\\
-x_{2,2}x_{2,1}e_{\sigma,2}+x_{2,2}x_{2,1}e_{\tau,2}&\in\Fb_4\\
-x_{2,3}x_{2,2}x_{1,1}e_{\alpha,2}+x_{2,3}x_{2,1}x_{1,2}e_{\sigma,2}&\in\Fb_4
\end{align*}
where the maps $\sigma,\tau,\alpha$ are given as
\[
\begin{minipage}{2cm}
\begin{align*}
\sigma:[2]&\to[4]\\
1&\mapsto1\\
2&\mapsto4
\end{align*}
\end{minipage}
\hspace{1cm}
\begin{minipage}{2cm}
\begin{align*}
\tau:[2]&\to[4]\\
1&\mapsto1\\
2&\mapsto3
\end{align*}
\end{minipage}
\hspace{1cm}
\begin{minipage}{2cm}
\begin{align*}
\alpha:[2]&\to[4]\\
1&\mapsto2\\
2&\mapsto4.
\end{align*}
\end{minipage}
\]
The script in \cite{M} then implements Theorem \ref{thm:MFR} to obtain the beginning of the minimal resolution of $\Mb$:
\[
\cdots \to \Fb^4\to\Fb^3\to\Fb^2\to\Fb^1\to\Fb^0\to\Mb\to 0
\]
where
\begin{align*}
\rank(\Fb^0)&=1\\
\rank(\Fb^1)&=2\\
\rank(\Fb^2)&=5\\
\rank(\Fb^3)&=9\\
\rank(\Fb^4)&=14.
\end{align*}
Finally, $\Fb^0$ is generated in width $3$, and if $1 \le  k \le 4$, then $\Fb^k$ has its generators in width $k+4$.
\item Let $\Fb=\Fb^{\OI,1}\oplus\Fb^{\OI,1}$ have basis $\{e_{\id_{[1]},1},e_{\id_{[1]},2}\}$ and let $c=2$ so that $\Pb$ has two rows of variables. Consider the submodule $\Nb$ of $\Fb$ generated by $x_{1,2}x_{1,1}e_{\pi,1}+x_{2,2}x_{2,1}e_{\rho,2}$,  where $\pi:[1]\to[2]$ is given by $1\mapsto2$ and $\rho:[1]\to[2]$ is given by $1\mapsto1$. As in the previous example, the Macaulay2 script in \cite{M} computes the beginning of the minimal resolution of $\Nb$:
\[
\cdots \to \Fb^4\to\Fb^3\to\Fb^2\to\Fb^1\to\Fb^0\to\Nb\to 0
\]
where
\begin{align*}
\rank(\Fb^0)&=1\\
\rank(\Fb^1)&=1\\
\rank(\Fb^2)&=2\\
\rank(\Fb^3)&=3\\
\rank(\Fb^4)&=4.
\end{align*}
Here, $\Fb^0$ is generated in width $2$, and if $1 \le  k \le 4$, then $\Fb^k$ has its generators in width $k+3$.
\end{enumerate}
\end{ex}

Finally, we remark that any free resolution $\Fb^\bullet$ of $\Mb$ induces, for any integer $w\ge 0$, a free resolution $\Fb^\bullet_w$ of $\Mb_w$ by restricting each differential width-wise. Thus, if we think of $\Mb$ as a sequence $(\Mb_n)_n$ of related modules, our techniques ``simultaneously'' compute a free resolution for each $\Pb_n$-module $\Mb_n$ in the sequence. Even if the resolution $\Fb^\bullet$ was minimal, its width $w$ component $\Fb^\bullet_w$ is not necessarily a minimal resolution of $\Mb_n$ over $\Pb_n$. However, if necessary, $\Fb^\bullet_w$ can be pruned to get a minimal free resolution of $\Mb_n$. Indeed, it is straightforward to adapt the pruning method of \cite{FN21} to this case. In fact, it becomes considerably simpler. 
\medskip


\end{document}